\documentclass[a4paper]{amsart}

\usepackage{amscd}
\usepackage{amsmath}
\usepackage{amssymb}
\usepackage{amsthm}
\usepackage{bbm}

\usepackage[T1]{fontenc}

\newif\ifpdf
\ifx\pdfoutput\undefined
   \pdffalse        
\else
   \pdfoutput=1     
   \pdftrue
\fi

\ifpdf
   \usepackage[pdftex]{graphicx}
\else
   \usepackage{graphicx}
\fi

\numberwithin{equation}{section} \swapnumbers

\newtheorem{satz}{Satz}[section]

\newtheorem{theorem}[satz]{Theorem}
\newtheorem{proposition}[satz]{Proposition}

\newtheorem{remark}[satz]{Remark}

\begin{document}

\title[Isomorphisms for spaces of predictable processes]{Isomorphisms for spaces of predictable processes and an extension of the It\^{o} integral}
\author{Barbara R\"udiger \and Stefan Tappe}
\thanks{We are grateful to Sergio Albeverio, Damir Filipovi\'{c}, Michael Kupper and Vidyadhar Mandrekar for their helpful remarks and discussions.}
\address{Bergische Universit\"{a}t Wuppertal, Fachbereich C -- Mathematik und Informatik, Gau{\ss}stra{\ss}e 20, D-42097 Wuppertal, Germany}
\email{ruediger@uni-wuppertal.de}
\address{Leibniz Universit\"{a}t Hannover, Institut f\"{u}r Mathematische Stochastik, Welfengarten 1, 30167 Hannover, Germany}
\email{tappe@stochastik.uni-hannover.de}
\begin{abstract}
Our goal of this note is to give an easy proof that spaces of predictable processes with values in a Banach space are isomorphic to spaces of progressive resp. adapted, measurable processes. This provides a straightforward extension of the It\^{o} integral in infinite dimensions. We also outline an application to stochastic partial differential equations.
\end{abstract}
\keywords{stochastic processes in infinite dimensions, isomorphisms for spaces of predictable processes, It\^{o} integral, stochastic partial differential equations}
\subjclass[2010]{60H15, 60G17}

\maketitle\thispagestyle{empty}

\section{Introduction}

The It\^{o} integral for predictable processes is well-established in the literature, see, e.g., \cite{Jacod-Shiryaev, Protter}. 
For non-predictable integrands, there is also an integration theory if the driving process is a continuous semimartingale, see, e.g., \cite{Karatzas-Shreve, Revuz, Atma-book}, and some references, such as \cite{Skorokhod, Liptser-Shiryaev, Barbara-Integration}, also consider the situation where the driving noise has jumps. 

Another approach to stochastic integration has been presented in \cite{Onno-Wiener, Onno-Levy}. Here the connection to the usual It\^{o} integral was established in \cite{Tappe}, see also Appendix B in \cite{Filipovic-Tappe}.

The goal of the note is to give an easy proof that spaces of predictable processes with values in a Banach space are isomorphic to spaces of progressive resp. adapted, measurable processes, which provides a straightforward extension of the It\^{o} integral for Banach space valued processes. We also compute the inverse of the embedding operator of these spaces in particular situations.

The remainder of this text is organized as follows: In Section \ref{sec-proj} we prove the announced result and compute the inverse of the embedding operator. In Section \ref{sec-integral} we consider the It\^{o} integral in various situations and sketch an application to stochastic partial differential equations.

\section{Isomorphisms for spaces of predictable processes}\label{sec-proj}

Let $(\tilde{\Omega},\tilde{\mathbb{P}},\tilde{\mathcal{F}})$ be a measure space. In view of our applications in Section \ref{sub-Prm}, we do not demand that $(\tilde{\Omega},\tilde{\mathbb{P}},\tilde{\mathcal{F}})$ is a probability space. Moreover, let $(\tilde{\mathcal{F}}_t)_{t \geq 0}$ be a filtration satisfying the usual conditions.

Fix $T > 0$ and let $\mu$ be a measure on $(\tilde{\Omega} \times [0,T], \tilde{\mathcal{F}}_T \otimes \mathcal{B}([0,T]))$ with marginales
\begin{align*}
\mu(A \times [0,T]) = \tilde{\mathbb{P}}(A), \quad A \in \tilde{\mathcal{F}}_T.
\end{align*}
We assume that there exists a sequence $(A_n)_{n \in \mathbb{N}} \subset \tilde{\mathcal{F}}_0$ such that $A_n \uparrow \tilde{\Omega}$ and $\tilde{\mathbb{P}}(A_n) < \infty$ for all $n \in \mathbb{N}$. In particular, the measures $\tilde{\mathbb{P}}$ and $\mu$ are $\sigma$-finite.

There exists a transition kernel $K : \tilde{\Omega} \times \mathcal{B}([0,T]) \rightarrow \mathbb{R}_+$ from $(\tilde{\Omega},\tilde{\mathcal{F}}_T)$ to $([0,T],\mathcal{B}([0,T]))$ such that
\begin{align*}
\mu(B) = \int_{\tilde{\Omega}} \int_0^T \mathbbm{1}_B(\tilde{\omega},t) K(\tilde{\omega},dt) \tilde{\mathbb{P}}(d \tilde{\omega}), \quad B \in \tilde{\mathcal{F}}_T \otimes \mathcal{B}([0,T]),
\end{align*}
see \cite[Sec. II.1a]{Jacod-Shiryaev}. 
We denote by $\tilde{\mathcal{P}}_T$ the predictable $\sigma$-algebra on $\tilde{\Omega} \times [0,T]$. Let $F$ be a separable Banach space. Fix an arbitrary $p \geq 1$ and define the spaces
\begin{align*}
L_{T,{\rm pred}}^p(F) &:= L^p(\tilde{\Omega} \times [0,T], \tilde{\mathcal{P}}_T, \mu; F),
\\ L_{T,{\rm prog}}^p(F) &:= L^p(\tilde{\Omega} \times [0,T], \tilde{\mathcal{F}}_T \otimes \mathcal{B}([0,T]), \mu; F) \cap {\rm Prog}_T(F),
\\ L_{T,{\rm ad}}^p(F) &:= L^p(\tilde{\Omega} \times [0,T], \tilde{\mathcal{F}}_T \otimes \mathcal{B}([0,T]), \mu; F) \cap {\rm Ad}_T(F),
\end{align*}
where ${\rm Prog}_T(F)$ denotes the linear space of all $F$-valued progressively measurable processes $(\Phi_t)_{t \in [0,T]}$ and ${\rm Ad}_T(F)$ denotes the linear space of all $F$-valued adapted processes $(\Phi_t)_{t \in [0,T]}$. Then we have the inclusions
\begin{align*}
L_{T,{\rm pred}}^p(F) \subset L_{T,{\rm prog}}^p(F) \subset L_{T,{\rm ad}}^p(F).
\end{align*}
In the upcoming theorem, we will show that these three spaces actually are isometrically isomorphic, provided the measures $A \mapsto K(\omega,A)$ are absolutely continuous. In particular, the latter two spaces are Banach spaces, too.

\begin{theorem}\label{thm-cong}
Suppose there is a nonnegative, measurable function $f : \tilde{\Omega} \times [0,T] \rightarrow \mathbb{R}$ such that for each $\tilde{\omega} \in \tilde{\Omega}$ we have $K(\tilde{\omega},dt) = f(\tilde{\omega},t)dt$. Then we have
\begin{align*}
L_{T,{\rm pred}}^p(F) \cong L_{T,{\rm prog}}^p(F) \cong L_{T,{\rm ad}}^p(F).
\end{align*}
\end{theorem}

\begin{proof}
It suffices to prove that for each $\Phi \in L_{T,{\rm ad}}^p(F)$ there exists a process $\pi(\Phi) \in L_{T,{\rm pred}}^p(F)$ such that $\Phi = \pi(\Phi)$ almost everywhere with respect to $\mu$.

Let $\Phi \in L_{T,{\rm ad}}^p(F)$ be arbitrary.
We will show that there is a sequence $(\Phi^n)_{n \in \mathbb{N}} \subset L_{T,{\rm pred}}^p(F)$ such that $\Phi^n \rightarrow \Phi$ in $L_{T,{\rm ad}}^p(F)$. Then, $(\Phi^n)_{n \in \mathbb{N}}$ is a Cauchy sequence in $L_{T,{\rm pred}}^p(F)$ and thus has a limit $\pi(\Phi) \in L_{T,{\rm pred}}^p(F)$. But this limit has the property $\Phi = \pi(\Phi)$ almost everywhere with respect to $\mu$, which will finish the proof.

The proof of the existence of a sequence $(\Phi^n)_{n \in \mathbb{N}} \subset L_{T,{\rm pred}}^p(F)$ satisfying $\Phi^n \rightarrow \Phi$ in $L_{T,{\rm ad}}^p(F)$ is divided into two steps:

\begin{enumerate}
\item First of all, we may assume that 
\begin{align}\label{measure-bounded}
\tilde{\mathbb{P}}(\tilde{\Omega}) = \mu(\tilde{\Omega} \times [0,T]) < \infty 
\end{align}
and that there is a constant $M > 0$ such that
\begin{align}\label{Phi-bounded}
\| \Phi \| \leq M \quad \text{everywhere.}
\end{align}
Indeed, by assumption, there exists a sequence $(A_n)_{n \in \mathbb{N}} \subset \tilde{\mathcal{F}}_0$ with $A_n \uparrow \tilde{\Omega}$ and $\tilde{\mathbb{P}}(A_n) < \infty$ for all $n \in \mathbb{N}$. Defining the sequence $(\Phi^n)_{n \in \mathbb{N}} \subset L_{T,{\rm ad}}^p(F)$ by $\Phi^n := (\Phi \wedge n) \mathbbm{1}_{A_n}$, Lebesgue's dominated convergence theorem yields that $\Phi^n \rightarrow \Phi$ in $L_{T,{\rm ad}}^p(F)$.

\item Now, we proceed with a similarly technique as in \cite[pp. 97--99]{Liptser-Shiryaev}. We extend $\Phi$ to a process $(\Phi_t)_{t \in \mathbb{R}}$ by setting
\begin{align*}
\Phi_t(\tilde{\omega}) := 0 \quad \text{for $(\tilde{\omega},t) \in \tilde{\Omega} \times \mathbb{R} \setminus [0,T]$.}
\end{align*}
Defining for $n \in \mathbb{N}$ the function $\theta_n : \mathbb{R} \rightarrow \mathbb{R}$ by
\begin{align*}
\theta_n(t) := \sum_{j \in \mathbb{Z}} \frac{j-1}{2^n} \mathbbm{1}_{(\frac{j-1}{2^n},\frac{j}{2^n}]}(t),
\end{align*} 
we have $\theta_n(t) \uparrow t$ for all $t \in \mathbb{R}$.
The shift semigroup $(S_t)_{t \geq 0}$, $S_t f = f(t + \cdot)$ is strongly continuous on $L^p(\mathbb{R};F)$. Thus,
performing integration by the substitution $t \rightsquigarrow t+s$, using Fubini's theorem, Lebesgue's dominated convergence theorem and noting (\ref{measure-bounded}), (\ref{Phi-bounded}) we obtain
\begin{align*}
&\int_{\tilde{\Omega}} \int_0^T \int_{0}^{T} \| \Phi_{s + \theta_n(t-s)}(\tilde{\omega}) - \Phi_t(\tilde{\omega}) \|^p dt ds \tilde{\mathbb{P}}(d \tilde{\omega}) 
\\ &= \int_{\tilde{\Omega}} \int_0^T \int_{-s}^{T-s} \| \Phi_{s + \theta_n(t)}(\tilde{\omega}) - \Phi_{s+t}(\tilde{\omega}) \|^p dt ds \tilde{\mathbb{P}}(d \tilde{\omega})
\\ &= \int_{\tilde{\Omega}} \int_0^T \int_{0}^{T-t} \| \Phi_{s + \theta_n(t)}(\tilde{\omega}) - \Phi_{s+t}(\tilde{\omega}) \|^p ds dt \tilde{\mathbb{P}}(d \tilde{\omega})
\\ &\quad + \int_{\tilde{\Omega}} \int_{-T}^0 \int_{-t}^{T} \| \Phi_{s + \theta_n(t)}(\tilde{\omega}) - \Phi_{s+t}(\tilde{\omega}) \|^p ds dt \tilde{\mathbb{P}}(d \tilde{\omega}) \rightarrow 0.
\end{align*}
After passing to a subsequence, if necessary, for $\tilde{\mathbb{P}} \otimes \lambda \otimes \lambda$--almost all $(\tilde{\omega},s,t) \in \tilde{\Omega} \times [0,T] \times [0,T]$ we have
\begin{align*}
\| \Phi_{s + \theta_n(t-s)}(\tilde{\omega}) - \Phi_{t}(\tilde{\omega}) \|^p \rightarrow 0,
\end{align*}
where $\lambda$ denotes the Lebesgue measure.
Thus, there exists $s \in [0,T]$ such that
\begin{align}\label{Lip-Shi}
\| \Phi_{s + \theta_n(t-s)}(\tilde{\omega}) - \Phi_{t}(\tilde{\omega}) \|^p \rightarrow 0 \quad \text{for $\tilde{\mathbb{P}} \otimes \lambda$--almost all $(\tilde{\omega},t) \in \tilde{\Omega} \times [0,T]$.}
\end{align}
For $n \in \mathbb{N}$ we define the process $\Phi^n = (\Phi_t^n)_{t \in [0,T]}$ by
\begin{align*}
\Phi_t^n := \Phi_{s + \theta_n(t-s)} = \sum_{j \in \mathbb{Z}} \Phi_{s + \frac{j-1}{2^n}} \mathbbm{1}_{(s + \frac{j-1}{2^n},s + \frac{j}{2^n}]}(t), \quad t \in [0,T].
\end{align*}
Note that $\Phi^n$ is predictable, because $\Phi$ is adapted.
Hence, we have $(\Phi^n)_{n \in \mathbb{N}} \subset L_{T,{\rm pred}}^p(F)$. By assumption, there is a nonnegative, measurable function $f : \tilde{\Omega} \times [0,T] \rightarrow \mathbb{R}$ such that for each $\tilde{\omega} \in \tilde{\Omega}$ we have $K(\tilde{\omega},dt) = f(\tilde{\omega},t)dt$. Using (\ref{measure-bounded}) we have
\begin{align*}
\int_{\tilde{\Omega}} \int_0^T f(\tilde{\omega},t) dt \tilde{\mathbb{P}}(d \tilde{\omega}) 
= \int_{\tilde{\Omega}} \int_0^T K(\tilde{\omega},dt) \tilde{\mathbb{P}}(d \tilde{\omega}) = \mu(\tilde{\Omega} \times [0,T]) < \infty.
\end{align*}
Noting (\ref{measure-bounded}), (\ref{Phi-bounded}), we obtain by (\ref{Lip-Shi}) and Lebesgue's dominated convergence theorem
\begin{align*}
&\iint_{\tilde{\Omega} \times [0,T]} \| \Phi^n - \Phi \|^p d \mu 
= \int_{\tilde{\Omega}} \int_0^T \| \Phi_{s + \theta_n(t-s)}(\tilde{\omega}) - \Phi_t(\tilde{\omega}) \|^p K(\tilde{\omega},dt) \tilde{\mathbb{P}}(d \tilde{\omega}) 
\\ &= \int_{\tilde{\Omega}} \int_0^T \| \Phi_{s + \theta_n(t-s)}(\tilde{\omega}) - \Phi_t(\tilde{\omega}) \|^p f(\tilde{\omega},t) dt \tilde{\mathbb{P}}(d \tilde{\omega}) \rightarrow 0,
\end{align*}
showing that $\Phi^n \rightarrow \Phi$ in $L_{T,{\rm ad}}^p(F)$.

\end{enumerate}
\end{proof}

\begin{remark}
Let $\pi : L_{T,{\rm ad}}^p(F) \rightarrow L_{T,{\rm pred}}^p(F)$ be the inverse of the embedding operator ${\rm Id} : L_{T,{\rm pred}}^p(F) \hookrightarrow L_{T,{\rm ad}}^p(F)$. Then,
for $\Phi \in L_{T,{\rm ad}}^p(F)$ the process $\pi(\Phi)$ coincides with the conditional expectation of $\Phi$ given the predictable $\sigma$-algebra $\tilde{\mathcal{P}}_T$, that is
\begin{align*}
\pi (\Phi) = \tilde{\mathbb{E}}[ \Phi \,|\, \tilde{\mathcal{P}}_T ],
\end{align*}
because $\pi(\Phi) = \Phi$ almost everywhere with respect to $\mu$.
\end{remark}

If the process $\Phi$ is c\`{a}dl\`{a}g, then $\pi(\Phi)$ is easy to determine:

\begin{proposition}\label{prop-cadlag}
Suppose there is a nonnegative, measurable function $f : \tilde{\Omega} \times [0,T] \rightarrow \mathbb{R}$ such that for each $\tilde{\omega} \in \tilde{\Omega}$ we have $K(\tilde{\omega},dt) = f(\tilde{\omega},t)dt$, and suppose that $\Phi \in \mathcal{L}_{T,{\rm ad}}^p(F)$ is c\`adl\`ag. Then we have $\pi(\Phi) = \Phi_-$.
\end{proposition}

\begin{proof}
First, we note that the process $\Phi_-$ is predictable.
Let $\tilde{\omega} \in \tilde{\Omega}$ be arbitrary. The set $\mathcal{N}_{\tilde{\omega}} = \{ t \in [0,T] : \Delta
\Phi_t(\tilde{\omega}) \neq 0 \}$ is countable. Hence, by the continuity of the measure $A \mapsto K(\tilde{\omega},A)$ we have
\begin{align*}
K(\tilde{\omega},\mathcal{N}_{\tilde{\omega}}) = 0 \quad \text{for all $\tilde{\omega} \in \tilde{\Omega}$.}
\end{align*}
Therefore, we obtain
\begin{align*}
&\iint_{\tilde{\Omega} \times [0,T]} \| \Phi_- \|^p d\mu = \int_{\tilde{\Omega}} \int_0^T \| \Phi_{t-}(\tilde{\omega}) \|^p K(\tilde{\omega},dt) \tilde{\mathbb{P}}(d\tilde{\omega}) 
\\ &= \int_{\tilde{\Omega}} \int_0^T \| \Phi_{t}(\tilde{\omega}) \|^p K(\tilde{\omega},dt) \tilde{\mathbb{P}}(d\tilde{\omega}) = \iint_{\tilde{\Omega} \times [0,T]} \| \Phi \|^p d\mu < \infty,
\end{align*}
because $\Phi \in \mathcal{L}_{T,{\rm ad}}^p(F)$, showing that $\Phi_- \in \mathcal{L}_{T,{\rm pred}}^p(F)$. Moreover, we get
\begin{align*}
&\iint_{\tilde{\Omega} \times [0,T]} \| \Phi - \Phi_- \|^p d\mu =
\int_{\tilde{\Omega}} \int_0^T \| \Phi_t(\tilde{\omega}) - \Phi_{t-}(\tilde{\omega}) \|^p K(\tilde{\omega},dt) \tilde{\mathbb{P}}(d\tilde{\omega}) 
\\ &=\int_{\tilde{\Omega}} \int_0^T \| \Delta \Phi_t(\tilde{\omega}) \|^p K(\tilde{\omega},dt) \tilde{\mathbb{P}}(d\tilde{\omega}) = 0,
\end{align*}
proving that $\Phi = \Phi_-$ almost everywhere with respect to $\mu$.
\end{proof}

\begin{remark}
Let $F := \mathbb{R}$, let
\begin{align*}
(\tilde{\Omega},\tilde{\mathcal{F}},(\tilde{\mathcal{F}}_t)_{t \geq 0},\tilde{\mathbb{P}}) := (\Omega,\mathcal{F},(\mathcal{F}_t)_{t \geq 0},\mathbb{P})
\end{align*}
be a filtered probability space and let $\mu$ be the product measure $\mu = \mathbb{P} \otimes K$, where $K$ is an absolutely continuous, finite measure on $([0, T], \mathcal{B}[0, T])$. Suppose that $\Phi \in \mathcal{L}_{T, \rm ad}^p(\mathbb{R})$ is a bounded process. According to \cite[Thm. VI.43]{Dellacherie}, there exist (up to an evanescent set) unique processes ${^o}\Phi$ and ${^p}\Phi$ such that ${^o}\Phi$ is optional (and hence progressively measurable), ${^p}\Phi$ is predictable and we have
\begin{align*}
\mathbb{E}[ X_{\tau} \mathbbm{1}_{ \{ \tau < \infty \} } \,|\, \mathcal{F}_{\tau} ] &= {^o} \Phi_{\tau} \mathbbm{1}_{\{ \tau < \infty \}} \quad \text{almost surely for every stopping time $\tau$,}
\\ \mathbb{E}[ X_{\tau} \mathbbm{1}_{ \{ \tau < \infty \} } \,|\, \mathcal{F}_{\tau -} ] &= {^p} \Phi_{\tau} \mathbbm{1}_{\{ \tau < \infty \}} \quad \text{almost surely for every predictable time $\tau$.}
\end{align*}
They are called the optional and the predictable projection of $\Phi$. By \cite[Remark VI.44.g]{Dellacherie} the optional projection ${^o} \Phi$ is a modification of $\Phi$, and hence
\begin{align*}
\mu(\Phi \neq {^o}\Phi) = (\mathbb{P} \otimes K)(\Phi \neq {^o}\Phi) = \int_0^T \mathbb{P}(\Phi_t \neq {^o}\Phi_t) K(dt) = 0.
\end{align*}
Moreover, by \cite[Thm. VI.46]{Dellacherie} we have
\begin{align*}
\{ {^o} \Phi \neq {^p} \Phi \} = \bigcup_{n \in \mathbb{N}} [\![ \tau_n ]\!],
\end{align*}
where $(\tau_n)_{n \in \mathbb{N}}$ is a sequence of stopping times and $[\![ \tau_n ]\!]$ denotes the graph
\begin{align*}
[\![ \tau_n ]\!] = \{ (\omega,\tau_n(\omega)) : \omega \in \Omega \}.
\end{align*}
Consequently, by the continuity of the measure $K$ we obtain
\begin{align*}
\mu({^o}\Phi \neq {^p}\Phi) = (\mathbb{P} \otimes K)({^o}\Phi \neq {^p}\Phi) \leq \sum_{n \in \mathbb{N}} \mathbb{E} [ K( \{ \tau_n \} ) ] = 0,
\end{align*}
showing that the inverse of the embedding operator is given by $\pi(\Phi) = {^p} \Phi$.
\end{remark}

Theorem \ref{thm-cong} provides a straightforward extension of the It\^{o} integral. Usually, one defines the It\^{o} integral as a continuous linear operator
\begin{align}\label{I-operator}
\mathcal{I} : L_{T,{\rm pred}}^2(F) \rightarrow M_T^2(G),
\end{align}
where $G$ is another separable Banach space and $M_T^2(G)$ denotes the Banach space of all $G$-valued square-integrable martingales $M = (M_t)_{t \in [0,T]}$. In fact, if $F$ and $G$ are Hilbert spaces, then the integral operator (\ref{I-operator}) is even an isometry. Using that $L_{T,{\rm ad}}^2(F) \cong L_{T,{\rm pred}}^2(F)$ according to Theorem \ref{thm-cong}, we can define the It\^{o} integral as
continuous linear operator
\begin{align}\label{extended}
\mathcal{I} : L_{T,{\rm ad}}^2(F) \rightarrow M_T^2(G).
\end{align} 
By localization, we can further extend the It\^{o} integral to all $F$-valued progressively measurable processes $\Phi = (\Phi_t)_{t \geq 0}$ such that $\tilde{\mathbb{P}}$--almost surely
\begin{align*}
\int_0^T \| \Phi_s \|^2 K(ds) < \infty \quad \text{for all $T > 0$,}
\end{align*}
and then the integral process $\mathcal{I}(\Phi)$ is a local martingale. We shall outline some concrete situations in the upcoming section.

\section{The It\^{o} integral for adapted, measurable processes}\label{sec-integral}

We shall now outline the extension of the It\^{o} integral in various situations. In what follows, $(\Omega,\mathcal{F},(\mathcal{F}_t)_{t \geq 0},\mathbb{P})$ denotes a filtered probability space satisfying the usual conditions.

\subsection{The It\^{o} integral with respect to martingales}

Let $M$ be a real-valued, square-integrable martingale.
Recall that the predictable quadratic variation $\langle M,M \rangle$ is the unique adapted, non-decreasing process such that $M^2 - \langle M,M \rangle$ is a martingale, see \cite[Thm. I.4.2]{Jacod-Shiryaev}. We assume that the quadratic variation $\langle M,M \rangle$ is absolutely continuous, which is in particular the case for L\'{e}vy processes. We set
\begin{align*}
(\tilde{\Omega},\tilde{\mathcal{F}},(\tilde{\mathcal{F}}_t)_{t \geq 0},\tilde{\mathbb{P}}) := (\Omega,\mathcal{F},(\mathcal{F}_t)_{t \geq 0},\mathbb{P}) \quad \text{and} \quad \mu := \mathbb{P} \otimes \langle M,M \rangle
\end{align*}
and let $F=G$ be a separable Hilbert space. Proceeding as in \cite[Sec. I.4d]{Jacod-Shiryaev}, we define the It\^{o} integral $\Phi \cdot M$ as the isometry (\ref{I-operator}). Using Theorem \ref{thm-cong}, we extend it to the isometry (\ref{extended}). The resulting It\^{o} integral coincides with the stochastic integral constructed in \cite{Liptser-Shiryaev}.

We remark that this construction is still possible if $F,G$ are separable Banach spaces of M-type 2 (see, e.g., \cite[Chap. 6]{Pisier}). Moreover, the martingale $M$ may even be infinite dimensional. Then, $\Phi$ has values in the space $L(F,G)$ of bounded linear operators from $F$ to $G$, see \cite{Pratelli}.

As pointed out in \cite[Sec. 18.4.1]{Liptser-Shiryaev-2}, for non-predictable integrands the just defined integral may not coincide with the pathwise Lebesgue-Stieltjes integral, provided the latter exists. For example, let $N$ be a standard Poisson process and let $M$ be the martingale $M_t = N_t - t$. By Proposition \ref{prop-cadlag}, the inverse of $N$ under the embedding operator is given by $\pi(N) = N_-$, and therefore we obtain the It\^{o} integral
\begin{align*}
(N \cdot M)_t = (N_- \cdot M)_t = \frac{1}{2} (N_t^2 - N_t) - \int_0^t N_{s-} ds,
\end{align*}
because the It\^{o} integral of the predictable process $N_-$ coincides with the pathwise Lebesgue-Stieltjes integral. On the other hand, we obtain the pathwise Lebesgue-Stieltjes integral
\begin{align*}
\int_0^t N_s dM_s =  \frac{1}{2} (N_t^2 + N_t) - \int_0^t N_s ds = N_t + (N \cdot M)_t,
\end{align*}
which cannot be a martingale, because $\mathbb{E}[N_t] = t$.

Thus, our extension of the It\^{o} integral may not coincide with the pathwise Lebesgue-Stieltjes integral, but it preserves the martingale property of the integral process, which makes it interesting for applications.

\subsection{The It\^{o} integral with respect to infinite dimensional Wiener processes}

Let $H,U$ be separable Hilbert spaces and let $Q \in L(U)$ be a
compact, self-adjoint, strictly positive linear operator. Then, there
exist an orthonormal basis $\{ e_j \}$ of $U$ and a bounded sequence
$(\lambda_j)$ of strictly positive real numbers such that
\begin{align*}
Qu = \sum_j \lambda_j \langle u,e_j \rangle e_j, \quad u \in U
\end{align*}
namely, the $\lambda_j$ are the eigenvalues of $Q$, and each $e_j$
is an eigenvector corresponding to $\lambda_j$.
The space $U_0 := Q^{1/2}(U)$, equipped with inner product
$\langle u,v \rangle_{U_0} := \langle Q^{-1/2} u,
Q^{-1/2} v \rangle_U$, is another separable Hilbert space
and $\{ \sqrt{\lambda_j} e_j \}$ is an orthonormal basis.

Let $W$ be a $Q$-Wiener process \cite[p. 86,87]{Da_Prato} such that $Q$ is a trace class operator, that is ${\rm tr}(Q) = \sum_j \lambda_j < \infty$.
We denote by $L_2^0 := L_2(U_0,H)$ the space of Hilbert-Schmidt
operators from $U_0$ into $H$, which, endowed with the
Hilbert-Schmidt norm
\begin{align*}
\| \Phi \|_{L_2^0} := \sqrt{\sum_j \lambda_j \| \Phi e_j \|^2},
\quad \Phi \in L_2^0
\end{align*}
itself is a separable Hilbert space. We set
\begin{align*}
(\tilde{\Omega},\tilde{\mathcal{F}},(\tilde{\mathcal{F}}_t)_{t \geq 0},\tilde{\mathbb{P}}) := (\Omega,\mathcal{F},(\mathcal{F}_t)_{t \geq 0},\mathbb{P}), \quad \mu := \mathbb{P} \otimes \lambda, \quad F := L_2^0, \quad G := H,
\end{align*}
where $\lambda$ denotes the Lebesgue measure.
Following \cite[Chap. 4.2]{Da_Prato}, we define the It\^{o} integral $\Phi \cdot W$ as the isometry (\ref{I-operator}). Using Theorem \ref{thm-cong}, we extend it to the isometry (\ref{extended}). The resulting It\^{o} integral coincides with the stochastic integral constructed in \cite{Atma-book}.

Analogously, we define the stochastic integral with respect to an infinite dimensional
L\'{e}vy process (see \cite{P-Z-book}) for adapted, measurable integrands.

\subsection{The It\^{o} integral with respect to compensated Poisson random measures}\label{sub-Prm}

Let $(E,\mathcal{E})$ be a measurable space which we assume to be a
\textit{Blackwell space} (see, e.g., \cite{Getoor}). We remark
that every Polish space with its Borel $\sigma$-field is a Blackwell
space. Let $N$ be a homogeneous Poisson random measure on $\mathbb{R}_+
\times E$, see \cite[Def. II.1.20]{Jacod-Shiryaev}. Then its compensator is of
the form $dt \otimes \beta(dx)$, where $\beta$ is a $\sigma$-finite measure
on $(E,\mathcal{E})$. We define the compensated Poisson random measure $q(dt,dx) := N(dt,dx) - \beta(dx) dt$ and set
\begin{align*}
(\tilde{\Omega},\tilde{\mathcal{F}},(\tilde{\mathcal{F}}_t)_{t \geq 0},\tilde{\mathbb{P}}) := (\Omega \times E,\mathcal{F} \times \mathcal{E},(\mathcal{F}_t \times \mathcal{E})_{t \geq 0},\mathbb{P} \otimes \beta),
\quad \mu := \tilde{\mathbb{P}} \otimes \lambda,
\end{align*}
where $\lambda$ denotes the Lebesgue measure,
and let $F=G$ be separable Hilbert spaces. Proceeding as in \cite[Sec. 4]{Applebaum}, we define the It\^{o} integral $\Phi \cdot q$ as the isometry (\ref{I-operator}). Using Theorem \ref{thm-cong}, we extend it to the isometry (\ref{extended}).

We remark that this construction is still possible if $F$ is a separable Banach space of M-type 2 (see, e.g., \cite[Chap. 6]{Pisier}). 
The resulting It\^{o} integral coincides with the stochastic integral constructed in \cite{Barbara-Integration}.

Moreover, such a construction is also possible on general separable Banach spaces, provided that the inequality 
\begin{align*}
\mathbb{E} \Bigg[ \bigg\| \int_0^t \int_E \Phi(s,x) q(ds,dx) \bigg\|^2 \Bigg] \leq K_{\beta} \mathbb{E} \bigg[ \int_0^T \int_E \| \Phi(s,x) \|^2 \beta(dx) ds \bigg]
\end{align*}
holds for all simple processes $\Phi$, with a constant $K_{\beta} > 0$ only depending on $\beta$, see \cite{Mandrekar-Ruediger-2009}.

\subsection{Stochastic partial differential equations}

Finally, we mention that we can use the extension of the stochastic integral provided in this paper in order to solve stochastic differential equations or even stochastic partial differential equations
\begin{align}\label{SPDE}
\left\{
\begin{array}{rcl}
dr_t & = & (A r_t + \alpha(t,r_t)) dt + \sigma(t,r_t) dW_t + \int_E \gamma(t,x,r_t) q(dt,dx) \medskip
\\ r_0 & = & h_0
\end{array}
\right.
\end{align}
on Hilbert spaces driven by an infinite dimensional Wiener process and a compensated Poisson random measure. Such equations have been studied, e.g., in \cite{Ruediger-mild, SPDE, Marinelli-Prevot-Roeckner}. In equation (\ref{SPDE}), the operator $A : \mathcal{D}(A) \subset H \rightarrow H$ denotes the generator of a strongly continuous semigroup.

Using our extension of the It\^{o} integral, under appropriate Lipschitz conditions on the vector fields we can prove the existence if a unique mild solution for (\ref{SPDE}) by performing a fixed point argument on appropriate spaces of progressively measurable processes.

\end{document}